\documentclass[11pt,leqno]{article}
\usepackage{amsmath, amscd, amsthm, amssymb, graphics, xypic, mathrsfs, setspace, fancyhdr, bm, pdfsync, enumitem, mathptmx}
\usepackage[usenames, dvipsnames, svgnames, table]{xcolor}
\usepackage[letterpaper,top=1.05in, bottom=1.05in, left=1.05in, right=1.05in]{geometry}
\usepackage[colorlinks=true,pagebackref=true]{hyperref}
\hypersetup{backref}

\newcommand{\Spec}{\operatorname{Spec}}

\newcommand{\isomto}{{\stackrel{\sim}{\;\longrightarrow\;}}}
\newcommand{\isomt}{{\stackrel{{\scriptscriptstyle{\sim}}}{\;\rightarrow\;}}}

\renewcommand{\hom}{\operatorname{Hom}}

\newcommand{\Z}{{\mathbb Z}}

\newcommand{\aone}{{\mathbb A}^1}

\newcommand{\gm}[1]{{{\mathbb G}_{m}^{#1}}}

\newcommand{\et}{\text{\'et}}

\newcommand{\Nis}{{\operatorname{Nis}}}
\newcommand{\Zar}{\operatorname{Zar}} 

\newcommand{\aff}{\operatorname{aff}}

\newcommand{\Sm}{\mathrm{Sm}}

\newcommand{\Singaone}{\operatorname{Sing}^{\aone}\!\!}
\newcommand{\op}[1]{\operatorname{#1}}

\renewcommand{\setminus}{\smallsetminus}

\newcommand{\Addresses}{{
\bigskip
\footnotesize

A.~Asok, Department of Mathematics, University of Southern California, 3620 S.~Vermont Ave., Los Angeles, CA 90089-2532, United States; E-mail address: asok@usc.edu
}}

\newcounter{intro}
\theoremstyle{plain}
\newtheorem{thm}{Theorem}[subsection]

\newtheorem{lem}[thm]{Lemma}
\newtheorem{cor}[thm]{Corollary}
\newtheorem{prop}[thm]{Proposition}
\newtheorem*{claim*}{Claim} 

\newtheorem*{thm*}{Theorem}
\newtheorem*{problem*}{Problem}

\newtheorem{thmintro}{Theorem}

\theoremstyle{definition}

\theoremstyle{remark}
\newtheorem{rem}[thm]{Remark}
\newtheorem{remintro}[thmintro]{Remark}

\newtheorem{entry}[thm]{}

\numberwithin{equation}{subsection}

\begin{document}
\pagestyle{fancy}
\renewcommand{\sectionmark}[1]{\markright{\thesection\ #1}}
\fancyhead{}
\fancyhead[LO,R]{\bfseries\footnotesize\thepage}
\fancyhead[LE]{\bfseries\footnotesize\rightmark}
\fancyhead[RO]{\bfseries\footnotesize\rightmark}
\chead[]{}
\cfoot[]{}
\setlength{\headheight}{1cm}

\author{Aravind Asok\thanks{AA was partially supported by NSF Awards DMS-1802060  and DMS-2101898.}}

\title{{\bf Affine representability of quadrics revisited}}
\date{}
\maketitle

\begin{abstract}
The quadric $\operatorname{Q}_{2n}$ is the ${\mathbb Z}$-scheme defined by the equation $\sum_{i=1}^n x_i y_i = z(1-z)$.  We show that $\operatorname{Q}_{2n}$ is a homogeneous space for the split reductive group scheme $\operatorname{SO}_{2n+1}$ over ${\mathbb Z}$.  The quadric $\operatorname{Q}_{2n}$ is known to have the ${\mathbb A}^1$-homotopy type of a motivic sphere and the identification as a homogeneous space allows us to give a characteristic independent affine representability statement for motivic spheres.  This last observation allows us to give characteristic independent comparison results between Chow--Witt groups, motivic stable cohomotopy groups and Euler class groups.
\end{abstract}

\section{Introduction}
Assume $k$ is an arbitrary commutative (unital) base ring.  Consider the hypersurface $\op{Q}_{2n}$ in ${\mathbb A}^{2n+1}_k := \Spec k[x_1,\ldots,x_{2n+1}]$ cut out by the equation
\[
\sum_{i=1}^n x_i x_{n+i} = x_{2n+1}(1-x_{2n+1});
\]
the quadric hypersurface so-defined is smooth over $\Spec k$.  Let $q_{2n+1}$ be the standard split quadratic form $\sum_i x_i x_{n+i} + x_{2n+1}^2$ in $2n+1$-variables, and let us provisionally write $\op{S}_{2n}$ for the quadric hypersurface $q_{2n+1} = 1$ in ${\mathbb A}^{2n+1}_k$.  

If $2$ is invertible in $k$, then the standard action of $\op{SO}_{2n+1}$ on ${\mathbb A}^{2n+1}_k$ as isometries preserving $q_{2n+1}$ yields, upon choice of a base-point, an isomorphism between $\op{S}_{2n}$ and the homogeneous space $\op{SO}_{2n+1}/\op{SO}_{2n}$.  Since $2$ is invertible in $k$, the quadric $\op{Q}_{2n}$ is isomorphic to $\op{S}_{2n}$ and is thus itself a homogeneous space for $\op{SO}_{2n+1}$ (see, e.g., \cite[Lemma 3.1.7]{AHWII}).  If $2$ is not a unit in $k$, then the quadric $\op{S}_{2n}$ fails to be smooth over $k$ and thus is neither isomorphic to $\op{Q}_{2n}$ nor a homogeneous space for $\op{SO}_{2n+1}$.  Nevertheless, the following result shows that $\op{Q}_{2n}$ is still isomorphic to $\op{SO}_{2n+1}/\op{SO}_{2n}$.

\begin{thmintro}[See Theorem~\ref{thm:quadricsinchar2}]
\label{thmintro:main}
Assume $k$ is a commutative ring and $n \geq 1$ is an integer.
\begin{enumerate}[noitemsep,topsep=1pt]
	\item There is an action of $\op{SO}_{2n+1}$ on $\op{Q}_{2n}$ such that taking the orbit through the point $x_0 \in \op{Q}_{2n}(k)$ given by $x_{2n+1} = 1$, $x_i= 0$, $1 \leq i \leq 2n$ yields a surjective smooth morphism $\op{SO}_{2n+1} \to \op{Q}_{2n}$ that factors through an $\op{SO}_{2n+1}$-equivariant isomorphism
    \[
    \varphi: \op{SO}_{2n+1}/\op{SO}_{2n} \isomto \op{Q}_{2n}.
    \]
\item The induced $\op{SO}_{2n}$-torsor $\op{SO}_{2n+1} \to \op{Q}_{2n}$ is Zariski locally trivial.
\end{enumerate}
\end{thmintro}

While the action of $\op{SO}_{2n+1}$ on $\op{S}_{2n}$ is classical, the action of $\op{SO}_{2n+1}$ on $\op{Q}_{2n}$ is slightly less transparent.  Once we have described the action, a sequence of standard algebro-geometric results about actions of group-schemes reduces the proof of Theorem~\ref{thmintro:main} to an elementary analysis of transitivity of actions of the special orthogonal groups in the spirit of \cite[II.10-11]{Dieudonne}.  Our main interest in the above result is due to its numerous concrete consequences, which stem from the fact that $\op{Q}_{2n}$ is ``well-behaved" from the standpoint of $\aone$-homotopy theory, for example the following result holds.

\begin{thmintro}[See Corollary~\ref{cor:affinerepresentability}]
	\label{thmintro:aonenaive}
	If $k$ is a field and $n \geq 1$ is an integer, then the scheme $\op{Q}_{2n}$ is $\aone$-naive; in particular, for any smooth affine $k$-scheme $X$, there is a canonical bijection
	\[
	\pi_0(\Singaone \op{Q}_{2n}(X)) \isomto [X,\op{Q}_{2n}]_{\aone},
	\]
	i.e., naive $\aone$-homotopy classes of maps coincide with ``true" $\aone$-homotopy classes of maps.
\end{thmintro}

\begin{remintro}
The notion of an $\aone$-naive space was introduced in \cite[Definition 2.1.1]{AHWII}, and the final statement follows directly from the preceding statement by appeal to general results.  If $2$ is invertible in $k$, then Theorem~\ref{thmintro:main} is the conjunction of \cite[Lemma 3.1.7, Theorem 4.2.2]{AHWII} if $k$ is infinite and \cite[Theorem 2.15]{AHWIII} if $k$ is finite.  If $2$ is not assumed invertible in $k$, then Theorem~\ref{thmintro:main} was known when $n \leq 3$ by appeal to various low-dimensional exceptional isomorphisms: $\op{Q}_2 \cong \op{SL}_2/\gm{}$, $\op{Q}_4 \cong \op{Sp}_4/(\op{Sp}_2 \times \op{Sp}_2)$ (again \cite[Theorem 4.2.2]{AHWII}), and $\op{Q}_6 \cong \op{G}_2/\op{SL}_3$ (see \cite[Theorem 2.3.5]{AHWOctonion}).  The case $n = 4$ can be analyzed by using an interpretation of $\op{Q}_8$ as the octonionic projective line. 
\end{remintro}

The variety $\op{Q}_{2n}$ arises naturally in the theory of complete intersection ideals \cite{MohanKumar} and provides a smooth affine model of the motivic sphere $\op{S}^{2n,n}$ in the Morel--Voevodsky $\aone$-homotopy category (see \cite[p. 111]{MV} for discussion of motivic spheres, and \cite[Theorem 2]{AsokDoranFasel} for a precise statement).  In \cite{AsokDasFasel2021}, these points of view were united to establish links between Bhatwadekar--Sridharan Euler class groups (after M. Nori) \cite{BS}, motivic stable cohomotopy groups, and Chow--Witt groups (as introduced by Barge--Morel \cite{BargeMorel}); we refer the reader to \cite{AsokDasFasel2021} for a more complete collection of references in this direction.   Theorem~\ref{thmintro:aonenaive} allows us to weaken the hypotheses in the main result of \cite[Theorem 1]{AsokDasFasel2021}, which we restate here for convenience (though we refer the reader to \cite{AsokDasFasel2021} for the relevant notation). 

\begin{thmintro}[Asok, Fasel]
	\label{thmintro:cohomotopy}
	Suppose $k$ is a field, $n$ and $d$ are integers, $n \geq 2$, and $X$ is a smooth affine $k$-scheme of dimension $d \leq 2n-2$.  
	\begin{enumerate}[noitemsep,topsep=1pt]
		\item The motivic cohomotopy set $[X,\op{Q}_{2n}]_{\aone}$ has a functorial abelian group structure, and 
		\item there is a functorial ``Hurewicz" homomorphism $[X,\op{Q}_{2n}]_{\aone} \to \widetilde{CH}^n(X)$, which is an isomorphism if $d \leq n$.
		\item If, furthermore, $k$ is infinite, then there is a functorial and surjective ``Segre class" homomorphism $s: \mathrm{E}^n(X) \to [X,\op{Q}_{2n}]_{\aone}$ from the Bhatwadekar--Sridharan Euler class groups to motivic cohomotopy;
		\item if $d \geq 2$, then the morphism $s$ is an isomorphism.
	\end{enumerate}
	In particular, under the hypotheses in \textup{Point (4)}, if $X$ is a smooth affine $k$-scheme of dimension $d$, then there is a functorial isomorphism
	\[
	\mathrm{E}^d(X) \isomto \widetilde{CH}^d(X).
	\]
\end{thmintro}

The proof of Theorem~\ref{thmintro:main} was inspired by analysis of the case $n = 3$ from \cite{AHWOctonion} mentioned above, which after various algebro-geometric reductions follows from \cite[Corollary 1.7.5]{SpringerVeldkamp}.  At the heart of those reductions is an interpretation of $\op{Q}_{2n}$ in terms of conditions on traces and norms of split octonion algebras.  This paper can be viewed in a similar vein: $\op{SO}_{2n+1}$ acts transitively on vectors in a representation space satisfying suitable ``norm" and ``trace" conditions.  The algebraic structures underlying the notions of ``norm" and ``trace" we use in this paper are quadratic Jordan algebras, and $\op{Q}_{2n}$ may be interpreted as a ``projective space" for such an algebra, in a sense that we explain momentarily.  While the theory of (quadratic) Jordan algebras make no appearance in the proofs, it does suggest various avenues of generalization, so we add a few comments about this point of view here.  

Assume $(J,1,U)$ is quadratic Jordan algebra over a commutative ring $k$; for us, $J$ is a finitely generated, projective $k$-module, $1$ is a distinguished element of $J$ and $U: J \to \operatorname{End}_k(J)$ ($x \mapsto U_x$) is a quadratic map satisfying various identities (see \cite[\S 1.2 Definition 3]{Jacobson} for the general definition; the cases of interest will even be obtained by base-change from $k = \Z$).  Of particular interest will be a class of quadratic Jordan algebras attached to quadratic spaces over $k$ called a (quadratic) spin factors; see \cite[Chapter 1.7]{Jacobson}.  The special orthogonal group-schemes $\op{SO}_{i}$ act by automorphisms on quadratic spin factors, even when $2$ is not invertible in $k$.

Given any quadratic Jordan algebra $(J,1,U)$ over a commutative ring $k$, the quadratic map $U$ allows one to speak of ``projection operators".  Indeed, if one defines $x^2 := U_x 1$, then a projection operator is one that satisfies $x^2 = x$.  Also attached to a quadratic spin factor is a suitable trace function, and one may consider ``rank 1" projection operators by imposing a suitable trace condition.  Granted these identifications, the scheme $\op{Q}_{2n}$ is precisely the projective space attached to a quadratic spin factor and $\op{SO}_{2n+1}$ acts by Jordan algebra automorphisms.

While the above interpretation does not aid in any calculations, it does suggest various natural generalizations of the main result.  For example, one can study  ``octonionic projective space" $\op{OP}^2 := \op{F}_4/\op{Spin}_9$ over an arbitrary commutative ring as a suitable ``projective space" of the exceptional quadratic Jordan algebra of Hermitian $3 \times 3$-matrices over the octonions.  Once again, this variety admits a description in terms of explicit ``rank 1" projection operators.  Because of the applications of Theorem~\ref{thmintro:main} stated above, we have decided to present a proof less encumbered by additional notation and terminology, and we defer possible generalizations to projective spaces of more general quadratic Jordan algebras to future work.

\subsubsection*{Acknowledgements}
The author would like to thank Konrad Voelkel for mentioning the Jordan algebra point of view on projective spaces to him many years ago, and Bob Guralnick and Skip Garibaldi for helpful references and useful discussions.  Finally, we thank Jean Fasel for many discussions around and corrections on preliminary versions of this note and Elden Elmanto for helpful comments.

\section{Orthogonal group actions on split quadrics}
Throughout this section we will assume that $k$ is an arbitary commutative base ring.  Section~\ref{ss:qspacesorthogonalgroups} recalls some basic facts about orthogonal groups in this generality.  Section~\ref{ss:orthogonalgroupactions} studies various properties of actions of reductive group schemes, and allows us to give a characteristic independent description of the action of $\op{SO}_{2n+1}$ on $\op{Q}_{2n}$.  Finally, Theorem~\ref{thmintro:main} is established in Section~\ref{ss:proofofmaintheorem}.  

\subsection{Quadratic spaces and orthogonal groups}
\label{ss:qspacesorthogonalgroups}
We begin by recalling some facts about orthogonal groups over $k$, in particular, we will not assume $2$ is invertible in $k$.  Assume $(V,q)$ is a quadratic space over $k$; we will always assume that $V$ is a projective $k$-module of constant rank, and that either $(V,q)$ is regular (a.k.a. non-singular) or, if $V$ has odd rank, semi-regular (see \cite[I.3.2]{Knus} for the former and \cite[\S IV.3.1]{Knus} for the latter). 

\subsubsection*{Special orthogonal groups}
We write $\op{O}(V,q)$ for the associated orthogonal group of isometries of $(V,q)$.  We will mostly be interested in the ``standard'' split quadratic spaces: 
\[
	q_{2n}(x_1,\ldots,x_{2n}) = \sum_{i=1}^n x_i x_{n+i}, \;\;\;\;\; \text{ and }  \;\;\;\;\; q_{2n+1}(x_1,\ldots,x_{2n+1}) = \sum_{i=1}^n x_i x_{n+1} + x_{2n+1}^2
\]
on the free modules of rank $2n$ and $2n+1$ over $k$.  We write $\op{O}_{i}$ for the orthogonal group scheme; functorially, this group scheme assigns to a $k$-algebra $R$ the usual orthogonal group $\op{O}(R^{\oplus i},q_{i})$ of automorphisms of $R^{\oplus i}$ preserving $q_{i}$.  We also consider the group-scheme $\op{GO}_i$ assigning to a $k$-algebra $R$ the group of orthogonal similitudes, i.e., the group of linear automorphisms of $R^{\oplus n}$ that preserve $q_i$ up to a unit.  

Abusing notation slightly, we write $\Z/2$ for the (constant) group scheme assigning to a $k$-algebra $R$ the additive group of continuous functions $\Spec R \to \Z/2$.  To discuss special orthgonal group schemes, recall that there is a Dickson invariant homomorphism
\[
D: \op{O}_i \longrightarrow \Z/2;
\]
(see \cite[IV.5.1]{Knus}) and a morphism $\chi: \Z/2 \to \mu_2$ defined by $f \mapsto (-1)^f$ such that the composite $\chi \circ D$ coincides with the determinant homomorphism $\det: \op{O}_i \to \mu_2$. If $i = 2n$, the Dickson invariant is split and surjective and one defines $\op{SO}_{2n}$ to be the kernel of the Dickson invariant, while if $i = 2n+1$, we define $\op{SO}_{2n+1}$ to be the kernel of the composite map $\chi \circ D$.  The group-scheme $\op{GO}_i$ has a subgroup scheme $\op{GSO}_i$ which is the fppf sub-group functor generated by $\gm{}$ and $\op{SO}_i$.  The group schemes $\op{O}_{2n}$ and $\op{SO}_{2n}$ are both smooth $k$-group schemes; the group scheme $\op{SO}_{2n+1}$ is a smooth $k$-group scheme, while $\op{O}_{2n+1}$ is a smooth $k$-group scheme if and only if $2$ is a unit in $k$ (see \cite[C.1.5]{Conrad} for these assertion).  Likewise, the group scheme $\op{GSO}_i$ is a smooth $k$-group scheme while $\op{GO}_i$ fails to be smooth if $i$ is odd (see, e.g., \cite[Remark C.3.11]{Conrad} and the discussion just preceding that statement).

We view $(k^{\oplus 2n+1},q_{2n+1})$ as a quadratic submodule of $(k^{\oplus 2n+2},q_{2n+2})$ by the embedding
\[
(e_1,\ldots,e_{2n+1}) \mapsto (e_1,\ldots,e_{n},e_{2n+1},e_{n+1},\ldots,e_{2n+1}).
\]
Likewise, we view $(k^{\oplus 2n},q_{2n})$ as the subspace of $(k^{\oplus 2n+1},q_{2n+1})$ where the coordinate function $x_{2n+1}$ vanishes.  These embeddings give rise to stabilization homomorphisms
\[
\op{O}_{i} \hookrightarrow \op{O}_{i+1}. \;\;\;\;\;\text{ and }\;\;\;\; \op{SO}_i \hookrightarrow \op{SO}_{i+1}
\] 
that we will need in the sequel.  
  
\subsubsection*{Bilinear forms, pointed quadratic spaces and traces}
If $q$ is a quadratic form, then we write $B$ for the associated bilinear form obtained by polarization, i.e., $B(x,y) = q(x + y) - q(x) - q(y)$.  When we consider $q_{i}$ on $k^{\oplus i}$ with coordinates $x_1,\ldots,x_{i}$, the associated bilinear forms are  
\[
\sum_{i=1}^n x_i y_{n+i} + x_{n+i}y_i, \;\;\;\;\;\text{ and }\;\;\;\;\;\sum_{i=1}^n x_i y_{n+i} + x_{n+i}y_i + 2 x_{2n+1} y_{2n+1}
\]
depending on whether $i$ is even or odd.

Recall that a {\em pointed} quadratic space over a commutative ring $k$ is a triple $(V,q,1)$ where $(V,q)$ is a quadratic space and $1 \in V$ such that $q(1) = 1 \in k$.  We point the split quadratic spaces $(k^{\oplus i},q_i)$ as follows: for $i = 2n$, we define $1 \in k^{\oplus 2n}$ to be $x_n = 1, x_{2n} = 1, x_i = x_{n+i} = 0, i = 1,\ldots,n-1$; for $i = 2n+1$, we define $1 \in k^{\oplus 2n+1}$ by $x_{2n+1} = 1$ and $x_i = 0, 1 \leq i \leq 2n$.  With this convention, only the embedding $(k^{\oplus 2n+1},q_{2n+1}) \to (k^{\oplus 2n+2},q_{2n+2})$ is pointed.  

Given a pointed quadratic space $(V,q,1)$, there is an associated trace function $t(x) := B(x,1)$.  In the special case $(k^{\oplus 2n+2},q_{2n+2},1)$, this trace function is given by the formula
\[
t_{2n+2}(x) = x_{n+1} + x_{2n+2}.
\]
Note that with this convention, $(k^{\oplus 2n+1},q_{2n+1})$ is contained in the orthogonal complement of the $k$-subspace $k \cdot 1$.  

\subsection{Orthogonal group actions}
\label{ss:orthogonalgroupactions}
We now analyze various orthogonal group actions; what we say is well-known over a field, but we review the statements over an arbitrary base ring for the convenience of the reader.  In order to streamline the analysis of actions over a general base, we begin by recalling some general facts about actions of reductive group schemes and homogeneous spaces by fiberwise techniques. 

\subsubsection*{Homogeneous spaces for reductive group schemes}
We recall the following result that allows us to deduce structural results about homogeneous spaces from ``fiberwise" computations over geometric points.  Recall that a reductive $k$-group scheme is a smooth affine $k$-group scheme $\op{G}$ such that the geometric fibers are {\em connected} reductive groups (e.g., \cite[Definition 3.1.1]{Conrad}). Unfortunately, we will need to deal with possibly disconnected group schemes and here the notion of geometric reductivity will be more useful for us; we refer the reader to \cite[Definition 9.1.1]{Alper} for a modern treatment of this notion.  

\begin{prop}
	\label{prop:checkisomorphism}
	Assume $\op{G}$ is an equidimensional (finitely presented) reductive $k$-group scheme with connected fibers and $(X,x)$ is an equidimensional, pointed, finitely presented smooth affine $k$-scheme equipped with an action of $\op{G}$.  Consider the orbit morphism:
	\[
	\tilde{\varphi}: \op{G} \longrightarrow X
	\]
	sending $g$ to $g \cdot x$.  Write $\op{G}_x$ for the stabilizer subgroup scheme of $\op{G}$.  If for every geometric point $s$ of $\Spec k$ the action of $\op{G}_{\kappa(s)}$ on $X_{\kappa(s)}$ is transitive, then
	\begin{enumerate}[noitemsep,topsep=1pt]
		\item $\tilde{\varphi}$ is finitely presented, faithfully flat and factors through a (pointed) isomorphism $\varphi: \op{G}/\op{G}_x \isomt X$ (in particular, $\op{G}/\op{G}_x$ exists as a scheme), and 
		\item the group scheme $\op{G}_x$ is a finitely presented flat affine group scheme that is moreover geometrically reductive; 
		\item if for every geometric point $s \in \Spec k$, the fiber $(\op{G}_x)_s$ is regular, then $\op{G}_x$ is a finitely presented smooth affine $k$-group scheme, $\op{G}_x/\op{G}_x^{\circ}$ is a finite \'etale $k$-group scheme, and $\op{G}_x^{\circ}$ is reductive. 
	\end{enumerate}
\end{prop}

\begin{proof}
	The first assertion is the implication (ii) $\Longrightarrow$ (i) from \cite[III \S 3 Proposition 2.1]{DemazureGabriel}.  We repeat the proof of the result to clarify and strengthen the conclusions.  Since $\op{G}$ and $X$ are both finitely presented over $k$, $\tilde{\varphi}$ is finitely presented \cite[1.6.2(v)]{EGAIV.1}.  In that case, since $\op{G}$ is $k$-smooth and affine by assumption as a reductive $k$-group scheme, the fibral flatness criterion of \cite[11.3.10]{EGAIV.3} states that $\tilde{\varphi}$ is flat if and only if the fibers of $\tilde{\varphi}$ at points $s$ of $\Spec k$ are flat.  Since flatness can be checked after faithfully flat extension, we reduce to considering the geometric fibers of $\tilde{\varphi}$.  
	
	Since $X$ is smooth, it is automatically reduced.  By generic flatness \cite[\href{https://stacks.math.columbia.edu/tag/052B}{Proposition 052B}]{stacks-project}, there is a dense open subscheme $U$ of $X$ such that $\tilde{\varphi}^{-1}(U) \to U$ is flat and finitely presented.  In that case, the transitivity of the action guarantees that we can cover $X$ by translates of $U$ and thus conclude that $\tilde{\varphi}$ is faithfully flat.  In conjunction with the discussion of the preceding paragraph, $\tilde{\varphi}$ is fppf.  Moreover the geometric fibers of $\tilde{\varphi}$ are isomorphic to $(\op{G}_x)_s$, so if they are regular, then $\tilde{\varphi}$ is smooth.
	
	Since $\op{G}_x \to \op{G}$ is a closed subgroup scheme by construction, we see that $\op{G}_x$ is a finitely presented, flat, affine $k$-group scheme.  In that case, the second assertion is ``Matsushima's theorem".  In more detail, by \cite[Theorem 9.4.1]{Alper} we conclude that $\op{G}_x$ is geometrically reductive group scheme.  In that case, note that $\op{G}_x/\op{G}_x^{\circ}$ is necessarily also a geometrically reductive $k$-group scheme.
	
	If $\op{G}_x$ happens to be smooth, which as we observed in the previous paragraph happens if and only if its geometric fibers are regular, then \cite[Theorem 9.7.6]{Alper} guarantees that $\op{G}_x$ is geometrically reductive if and only if $\op{G}_x^{\circ}$ is reductive and $\op{G}_x/\op{G}_x^{\circ}$ is finite.  In that case, it is automatically also smooth and therefore \'etale. 
\end{proof}

\begin{rem}
	In Proposition~\ref{prop:checkisomorphism}, to guarantee smoothness of the fibers, the hypothesis that geometric fibers of $\op{G}_x$ are smooth is essential.  For example, the morphism $\gm{} \to \gm{}$ given by $t \mapsto t^n$ is transitive at the level of geometric points and makes the source into the total space of a $\mu_n$-torsor over the target.  Thus, if $k$ is a field having characteristic $p$ with $p|n$, the morphism in question is flat but not smooth since it has non-reduced geometric fibers. 
\end{rem}

\subsubsection*{Checking transitivity}
In order to apply Proposition~\ref{prop:checkisomorphism} we need some group-theoretic facts to aid in checking transitivity at geometric points, especially in the case of orthogonal group actions.  We review some distinguished classes of elements in orthogonal group schemes.  

\begin{entry}[Reflections]
	\label{entry:reflections}
	Suppose $(V,q)$ is a quadratic space over a commutative ring $k$, and $B$ is the $k$-bilinear form obtained from $q$ by polarization.  Assume that $v \in V$ is an element such that $q(v) \in k^{\times}$.  In that case, the {\em reflection} $r_{v}$ is defined by the formula
	\[
	r_v(w) = w - q(v)^{-1}B(v,w)v.
	\]
	The reflection $r_v$ is an element of $\op{O}(V,q)$.  Note that if $(V,q) = (k^{\oplus 2n},q_{2n})$, then the Dickson invariant of $r_v$ is equal to $1$, while if $(V,q) = (k^{\oplus 2n+1},q_{2n+1})$ the element $r_v$ has determinant $-1$ \cite[\S 4.1.1]{Bass}.
\end{entry}

\begin{lem}[{\cite[Lemma 8.2]{EKM}}]
	\label{lem:reflectiontransitivity}
	Suppose $(V,q)$ is a quadratic space over a field $k$, and $B$ is the $k$-bilinear form obtained from $q$ by polarization.  Suppose $x$ and $y$ are elements of $v$ such that $q(x) = q(y)$.  
	\begin{enumerate}[noitemsep,topsep=1pt]
	\item If $q(x-y)$ is non-zero, then $r_{x-y}(x) = y$.
	\item If $q(x-y) = 0$, $w$ is a vector such that $q(w)$,$B(x,w)$ and $B(y,w)$ are simultaneously non-zero, then setting $w' = x - r_w y$, $q(w') \neq 0$ and $(r_{w} \circ r_{w'})(x) = y$.
	\end{enumerate}
\end{lem}

\subsubsection*{Representation spaces}
Consider the standard action of $\op{GSO}_{2n+2}$ on ${\mathbb A}^{2n+2}_k$.  This action preserves $q_{2n+2} = 0$ by definition and there is an induced action of $\op{GSO}_{2n+2}$ on the smooth affine $k$-scheme ${\mathbb A}^{2n+2}_k \setminus \{q_{2n+2} = 0\}$.  The base-point $1$ defines a $k$-point of ${\mathbb A}^{2n+2}_k$ and the next result yields an embedding of $\op{O}_{2n+1}$ in $\op{GSO}_{2n+2}$ as the subgroup stabilizing this point.

\begin{prop}
	\label{prop:oddorthogonalaction}
	The action of $\op{GSO}_{2n+2}$ on ${\mathbb A}^{2n+2}_k \setminus\{ {q_{2n+2} = 0} \}$ is transitive on geometric points and taking the orbit through $1$,  there is an induced isomorphism of schemes 
	\[
	\op{GSO}_{2n+2}/\op{O}_{2n+1} \isomto {\mathbb A}^{2n+2}_k \setminus\{ {q_{2n+2} = 0} \}.
	\] 
\end{prop}

\begin{proof}
	Suppose $x$ is a geometric point of ${\mathbb A}^{2n+2}_k \setminus\{ {q_{2n+2} = 0} \}$.  Henceforth, we write $q$ for $q_{2n+2}$.  By construction $q(x) \neq 0$ and $q(1)  = 1 \neq 0$.  Since $k$ is algebraically closed, we can always find $\lambda \in k^{\times}$ such that $q(\lambda x) = q(1) = 1$.  In that case, if $q(\lambda x-1) \neq 0$, then the reflection $r_{\lambda x-1}$ moves $\lambda x$ to $1$ (Lemma~\ref{lem:reflectiontransitivity}.1), and the composite of scaling by $\lambda$ and $r_{\lambda x - 1}$ lies in $\op{GSO}_{2n+2}$.
	
	Thus, suppose that $q(\lambda x - 1) = 0$.  In that case, since $\lambda x \neq 0$, we know that $\lambda x$ has a non-zero coordinate, and taking the standard basis vector $e_i$ for suitable $i$ we see that $B(\lambda x,e_i)$ is non-zero, i.e., the locus $B(\lambda x,-) \neq 0$ is a non-empty open subscheme of ${\mathbb A}^{2n+2}_k$.  Likewise, the locus $B(1,-) \neq 0$ is a non-empty open subscheme of ${\mathbb A}^{2n+2}_k$.  Since $k$ is algebraically closed and thus infinite, it follows that the intersection of $B(\lambda x,-) \neq 0$ and $B(1,-) \neq 0$ with $q(-) \neq 0$ is non-empty.  Therefore, by Lemma~\ref{lem:reflectiontransitivity}.2 there exists a composite of reflections taking $\lambda x$ to $1$.
	
	Appealing to the first two points of Proposition~\ref{prop:checkisomorphism}, there is an induced isomorphism 
	\[
	\op{GSO}_{2n+2}/\op{Stab}_1 \isomto {\mathbb A}^{2n+2}_k \setminus\{ {q_{2n+2} = 0} \},
	\]  
	and it remains to identify $\op{Stab}_1$, which is a geometrically reductive flat affine $k$-group scheme.  To this end, we analyze the fibers of $\op{Stab}_1$ over points of $\Spec k$.  Note that any element of $\op{GSO}_{2n+2}$ that stabilizes $1$ necessarily fixes the restriction of $q$ to the orthogonal complement to $k \cdot 1$.  In other words, any element of $\op{GSO}_{2n+2}$ that stabilizes $1$ preserves the form $q_{2n+1}$ and thus lies in $\op{O}_{2n+1}$.  Conversely any element of $\op{O}_{2n+1}$ stabilizes $1$, so we conclude that $\op{Stab}_1 = \op{O}_{2n+1}$.  Note: we cannot appeal to the third point of Proposition~\ref{prop:checkisomorphism} as $\op{O}_{2n+1}$ fails to be a smooth group scheme.  
\end{proof}

\begin{rem}
	In line with the discussion of the introduction, Proposition~\ref{prop:oddorthogonalaction} is a special case of a more general statement relating the structure group of quadratic Jordan algebra and the automorphism group.  General results of \cite[Corollary 6.6]{LoosAlgGp} guarantee that the structure group of a separable, unital, quadratic Jordan algebra is a reductive group scheme.  The automorphism group of the Jordan algebra can be identified with the subgroup scheme of the structure group scheme that preserves the identity.  The relationship between the structure group and the automorphism group for quadratic Jordan algebras was analyzed in \cite[\S 14]{Springer} and we refer the reader to \cite[\S 3]{BLG} for  more general results in the spirit of the above proposition.  
\end{rem}

\subsection{Proof of Theorem~\ref{thmintro:main}}
\label{ss:proofofmaintheorem}
Consider the pointed quadratic space $(k^{\oplus 2n+2},q_{2n+2},1)$.  Recall that the trace form $t_{2n+2}(x) = B_{2n+2}(x,1)$ where $B_{2n+2}$ is the bilinear form obtained from $q_{2n+2}$ by polarization.  Explicitly, this trace form is given by $t_{2n+2}(x) = x_{n+1} + x_{2n+2}$.  Proposition~\ref{prop:oddorthogonalaction} shows that the stabilizer of $1$ in $\op{GSO}_{2n+2}$ is identified with $\op{O}_{2n+1}$.  It follows that there is an action of $\op{SO}_{2n+1} \subset \op{O}_{2n+1}$ on ${\mathbb A}^{2n+2}_k$ as automorphisms that preserve $q_{2n+2}$ and fix $1$.  This $\op{SO}_{2n+1}$-action then preserves the hypersurface $t_{2n+2}(x) = 1$ since it fixes $1$ and thus induces an action on the variety defined by $q_{2n+2}(x) = 0$ and $t_{2n+2}(x) = 1$. After renaming variables appropriately, the variety defined by $q_{2n+2} = 0, t_{2n+2} = 1$ is precisely $\op{Q}_{2n}$.  We summarize these observations in the following result.

\begin{lem}
	\label{lem:constructingtheaction}
The closed subscheme of ${\mathbb A}^{2n+2}_k$ defined by $t_{2n+2} = 1$ and $q_{2n+2} = 0$ is isomorphic to the scheme $\op{Q}_{2n}$; this scheme comes equipped with the action of $\op{SO}_{2n+1}$ as a subgroup of $\op{GSO}_{2n+2}$ stabilizing the base-point $1 \in {\mathbb A}^{2n+2}_k$.
\end{lem}  

We consider the point $x_0 \in \op{Q}_{2n}(k)$ given by $x_{2n+2} = 1$, $x_1=\cdots=x_{2n+1} = 0$, and we view $\op{Q}_{2n}$ as pointed by $x_0$.  Observe that under the inclusion $\op{O}_{2n} \hookrightarrow \op{O}_{2n+1}$, there is an inclusion
\[
\op{O}_{2n} \hookrightarrow \op{Stab}_{x_0}(\op{O}_{2n+1}).
\]
Our goal will be to use this fact to identify the stabilizer in $\op{SO}_{2n+1}$ of the point $x_0$.  The next result yields Theorem~\ref{thmintro:main} from the introduction.

\begin{thm}
	\label{thm:quadricsinchar2}
	Assume $k$ is a commutative ring and $n \geq 1$ is an integer.  Consider the action of $\op{SO}_{2n+1}$ on $\op{Q}_{2n}$ of \textup{Lemma~\ref{lem:constructingtheaction}}.
	\begin{enumerate}[noitemsep,topsep=1pt]
		\item Taking the $\op{SO}_{2n+1}$-orbit through the $k$-point $x_0$ on $\op{Q}_{2n}$ yields a surjective smooth morphism 
		\[
		\tilde{\varphi}: \op{SO}_{2n+1} \to \op{Q}_{2n}
		\] 
		that factors through an $\op{SO}_{2n+1}$-equivariant isomorphism
		\[
		\varphi: \op{SO}_{2n+1}/\op{SO}_{2n} \isomto \op{Q}_{2n}.
		\]
		\item The $\op{SO}_{2n}$-torsor $\op{SO}_{2n+1} \to \op{Q}_{2n}$ is Zariski locally trivial.
	\end{enumerate}
\end{thm}

\begin{proof}
	Since all schemes and groups in question are defined over $\Z$, it suffices to prove the result in that case and the results in the general case are deduced by base change.  Thus, we assume $k = \Z$ in what follows.  By appeal to Proposition~\ref{prop:checkisomorphism}, if $\op{SO}_{2n+1}$ acts transitively on $\op{Q}_{2n}$ after base-change to geometric points of $\Z$, then there is an induced isomorphism $\op{SO}_{2n+1}/\op{Stab}_{x_0} \isomto \op{Q}_{2n}$, where $\op{Stab}_{x_0}$ is the stabilizer group-scheme of the point $x_0$ in $\op{SO}_{2n+1}$.  Transitivity of the action at geometric points is Proposition~\ref{prop:transitivityI} below.
	
	Now, we identify the stabilizer explicitly.  Since any element of the stabilizer fixes $x_0$, it preserves $k \cdot x_0$.  An explicit computation shows that one has a sequence of inclusions $\op{SO}_{2n} \subset \op{Stab}_{x_0} \subset \op{O}_{2n}$.  Suppose $\op{Stab}_{x_0}$ were all of $\op{O}_{2n}$.  In that case, the sequence of inclusions $\op{SO}_{2n} \subset \op{O}_{2n} \subset \op{SO}_{2n+1}$ would yield an fppf $\Z/2$-torsor of the form $\op{SO}_{2n+1}/\op{SO}_{2n} \to \op{Q}_{2n}$.  Since $\op{SO}_{2n+1}$ is connected, $\op{SO}_{2n+1}/\op{SO}_{2n}$ must be a connected affine scheme, so this torsor is necessarily non-trivial.  On the other hand, this torsor is necessarily \'etale locally trivial since it is an associated fiber space for an $\op{O}_{2n}$-torsor, and $\op{O}_{2n}$ is a smooth $k$-group scheme (see \cite[XVII.8.1]{SGA3.3} or \cite[Remarques 11.8.2]{BrauerIII}).  In that case, since this $\Z/2$-torsor is trivial upon base-change to the geometric generic point of $\Spec \Z$ by appeal to Lemma~\ref{lem:no2covers} below, and since $\Z/2$ is a constant group scheme, it must have been trivial to begin.  It follows that $\op{Stab}_{x_0}$ must be isomorphic to $\op{SO}_{2n}$ and thus $\op{Q}_{2n}$ is isomorphic to $\op{SO}_{2n+1}/\op{SO}_{2n}$.
	
	For the third statement, recall that Witt cancellation holds for quadratic spaces over local rings.  More precisely, assume $R$ is a local ring, and consider a morphism $\Spec R \to \op{SO}_{2n+1}/\op{SO}_{2n}$.  Such a morphism corresponds to an $\op{SO}_{2n}$-torsor on $\Spec R$ that becomes trivial when viewed as an $\op{SO}_{2n+1}$-torsor, i.e., a stably hyperbolic quadratic form.  By \cite[Corollary III.4.3]{Baeza}, if $q_1$ is a stably hyperbolic quadratic space over $R$, it follows that $q_1$ is actually hyperbolic, i.e., the morphism $\Spec R$ lifts to $\op{SO}_{2n+1}$.  It follows that $\op{SO}_{2n+1} \to \op{Q}_{2n}$ is Zariski locally trivial, since it has Zariski local sections.
\end{proof}

\begin{lem}
	\label{lem:no2covers}
	If $k$ is an algebraically closed field having characteristic not equal to $2$, then for any integer $n > 0$, the variety $\op{Q}_{2n}$ has no non-trivial \'etale $\Z/2$-torsors, i.e., $\op{H}^1_{\et}(\op{Q}_{2n},\Z/2) = 0$.
\end{lem}

\begin{proof}
	Under the hypothesis on $k$, the Kummer sequence identifies $\op{H}^1_{\et}(\op{Q}_{2n},\Z/2)$ with the $2$-torsion subgroup of $\op{Pic}(\op{Q}_{2n})$.  A straightforward induction argument using the localization sequence and the fact that $\op{Q}_{2n}$ has an open subscheme isomorphic to ${\mathbb A}^{2n-1} \times \gm{}$ with closed complement $\op{Q}_{2n-2} \times \aone$ then identifies $\op{Pic}(\op{Q}_{2n})$ as $\Z$ if $n = 1$ and $0$ if $n > 1$.  In either case, it follows that the $2$-torsion subgroup is trivial.
\end{proof}

\begin{prop}
	\label{prop:transitivityI}
	Suppose $k$ is an infinite field and $n > 0$ is an integer.  The action of $\op{SO}_{2n+1}(k)$ on $\op{Q}_{2n}(k)$ is transitive.  
\end{prop}

\begin{rem}
	The case $n = 0$ is exceptional because $\op{Q}_{0}$ is a disconnected scheme isomorphic to the disjoint union of two copies of $\Spec k$.  
\end{rem}

\begin{proof}
	If $k$ has characteristic not equal to $2$, then this fact is well-known, so let us assume that $k$ has characteristic $2$.  In that case, note that the map $\op{SO}_{2n+1}(k) \to \op{O}_{2n+1}(k)$ is a bijection.   Suppose $x \in \op{Q}_{2n}(k)$ is an arbitrary point not equal to $x_0$.  To establish the result, it therefore suffices to construct an element of $\op{O}_{2n+1}(k)$ that takes $x$ to $x_0$, which we will accomplish using explicit reflections.  For the remainder of this proof, we write $B$ for the bilinear form obtained by polarizing $q_{2n+2}$ and $t$ for the associated trace function $t_{2n+2}$. \newline
	
	\noindent {\bf Case 1}.  Suppose $q_{2n+2}(x - x_0) \neq 0$.  Since $t(x-x_0) = t(x) - t(x_0) = 0$, it follows that $x-x_0$ lies in the orthogonal complement of the linear space $k\cdot 1$.  In other words, the vector $x-x_0$ lies in the subspace $k^{\oplus 2n+1}$.  In that case, the reflection $r_{x-x_0}$ sends $x$ to $x_0$ (Lemma~\ref{lem:reflectiontransitivity}.1).  \newline
	
	\noindent {\bf Case 2.} Suppose $q_{2n+2}(x-x_0) = 0$.  We claim that we can choose $a$ such that $t(a) = 0$ and such that $q_{2n+2}(a) \neq 0$, $B(x,a) \neq 0$ and $B(x_0,a) \neq 0$.
	
	The locus of points where $t = 0$ corresponds to imposing the equation $x_{n+1} = x_{2n+2}$, which defines a closed subscheme of ${\mathbb A}^{2n+2}_k$ isomorphic to ${\mathbb A}^{2n+1}_k$.  The restriction of $q_{2n+2}$ to this subspace is $q_{2n+1}$.  In particular, the locus where $q_{2n+1} \neq 0$ is a principal open subset of ${\mathbb A}^{2n+1}_k$; this open subset is non-empty since $n \geq 1$ (the point $x_1 = x_{n+2} = 1$, $x_i = 0$ otherwise works).  
	
	Since $1 \in k^{\oplus 2n+2}$ has trace zero, and since $x$ and $x_0$ both lie on $\op{Q}_{2n}(k)$, we know that $1 = t(x_0) = B(x_0,1) = B(x,1) = t(x)$.  In other words, the Zariski open subsets of the hypersurface $t = 0$ defined by intersection with $B(x,-) \neq 0$ or $B(x_0,-) \neq 0$ are themselves non-empty. Since $k$ is infinite, their intersection is non-empty and likewise the intersection with the open subset $\{q_{2n+1} \neq 0\}$ is non-empty.  
	
	Set $a' = x - r_{a}(x_0)$. In that case, 
	\[
	t(a') = t(x) - t(r_a(x_0)) = q(a)^{-1}B(a,1)t(a) = 0
	\]
	as well.  The composite $r_a \circ r_{a'}$ takes $x$ to $x_0$ (by Lemma~\ref{lem:reflectiontransitivity}.2) as required.
	
	\end{proof}

	Over fields of positive characteristic there are alternative group-theoretic arguments for transitivity, once the action above has been defined and the stabilizer computed.  The following argument for transitivity was suggested by R. Guralnick.  

\begin{prop}
	Over any algebraically closed field $k$ of positive characteristic, the action of $\op{SO}_{2n+1}(k)$ on $\op{Q}_{2n}(k)$ is transitive.
\end{prop}

\begin{proof}
	To check transitivity over a given algebraically closed field, it suffices to check transitivity over an algebraically closed subfield by \cite[Proposition 1.1]{GLMS}.  In that case we may assume that $k$ is the algebraic closure of a finite field.  It then suffices to check transitivity over any finite field, and we do this by counting points.  We can compute the number of points of $\op{Q}_{2n}$ over a finite field inductively.  Indeed, for every $n \geq 1$, the scheme $\op{Q}_{2n}$ has an open subscheme isomorphic to ${\mathbb A}^{2n-1} \times \gm{}$ with closed complement $\op{Q}_{2n-2} \times \aone$.  Since $\op{Q}_0 = \Spec k \sqcup \Spec k$, one sees immediately that $|\op{Q}_2({\mathbb F}_q)| = q^2 + q$.  The decomposition above gives the recursive formula
	\[
	|\op{Q}_{2n}({\mathbb F}_q)| = q^{2n-1}(q-1) + q(|\op{Q}_{2n-2}({\mathbb F}_q)|),
	\]
	and a straightforward induction argument allows one to conclude that $|\op{Q}_{2n}({\mathbb F}_q)| = q^{2n}+q^n$.  
	
	On the other hand, the formulas for the order of the special orthogonal group over a finite field \cite[Theorem 25 on p. 77]{Steinberg} allow one to compute that $|\op{SO}_{2n+1}({\mathbb F}_q)/\op{SO}_{2n}({\mathbb F}_q)| = q^{2n} + q^n$ as well.  It follows immediately that the action of $\op{SO}_{2n+1}({\mathbb F}_q)$ on $\op{Q}_{2n}({\mathbb F}_q)$ is transitive for every $q = p^n$, and thus transitivity holds after passing to $\overline{{\mathbb F}_q}$ as well; the result follows.
\end{proof}

\section{Applications}
\label{s:consequences}
In this section, we deduce a selection of consequences of Theorem~\ref{thm:quadricsinchar2}.  The results of this section assume familiarity with motivic homotopy theory, in particular, the results of \cite{AHW,AHWII,AHWIII}.  

\subsection{Affine representability}
For the convenience of the reader, we recall that $\Sm_k$ is the category of smooth $k$-schemes, $\Sm_k^{\aff}$ is the subcategory of $\Sm_k$ consisting of affine schemes, and $\op{sPre}(\op{Sm}_k)$ is the category of simplicial presheaves on $\op{Sm}_k$.  If $t$ is a Grothendieck topology on $\Sm_k$, then $R_t$ is a fibrant replacement functor for the injective $t$-local model structure on $\op{sPre}(\op{Sm}_k)$ (see \cite[\S 3.1]{AHW} for more details), while $\Singaone$ is the singular construction (see \cite[\S 4.1]{AHW}).  We write $\op{Ho}(k)$ for the Morel--Voevodsky $\aone$-homotopy category as discussed in \cite[\S 5]{AHW} and for $X,Y \in \Sm_k$, we write $[X,Y]_{\aone}$ for $\hom_{\op{Ho}(k)}(X,Y)$.  First, we establish Theorem~\ref{thmintro:aonenaive} about affine representability of $\op{Q}_{2n}$ over any field.  

\begin{cor}
	\label{cor:affinerepresentability}
Assume $k$ is a field.
\begin{enumerate}[noitemsep,topsep=1pt]
\item The simplicial presheaf $R_{\Zar} \Singaone \op{Q}_{2n}$ is Nisnevich local and $\aone$-invariant.
\item If $X \in \Sm_k^{\aff}$, then the canonical map
\[
\pi_0(\Singaone \op{Q}_{2n}(X)) \longrightarrow [X,\op{Q}_{2n}]_{\aone}
\]
is an isomorphism.
\end{enumerate}
\end{cor}

\begin{proof}
Combine Theorem~\ref{thm:quadricsinchar2} and \cite[Theorem 2.6]{AHWIII}.
\end{proof}

\begin{rem}
	In fact, it seems likely that Corollary~\ref{cor:affinerepresentability} will extend to the case $k = \Z$.  Indeed, this would follow immediately if one knew the Bass--Quillen conjecture for Nisnevich locally trivial $\op{SO}_{2n+1}$-torsors for smooth $\Z$-algebras, i.e., if for any smooth $\Z$-algebra $A$, and any integer $i \geq 0$, the map 
	\[
	H^1_{\Nis}(\Spec A,\op{SO}_{2n+1}) \longrightarrow H^1_{\Nis}(\Spec A[x_1,\ldots,x_i],\op{SO}_{2n+1})
	\]
	is an isomorphism.
\end{rem}

\subsection{Euler class groups and motivic stable cohomotopy}
In this section, we establish Theorem~\ref{thmintro:cohomotopy} from the introduction and a further application to weak Euler class groups.

\begin{proof}[Proof of Theorem~\ref{thmintro:cohomotopy}]
    That the set $[X,\op{Q}_{2n}]_{\aone}$ has a functorial abelian group structure under the stated hypotheses on the dimension of $X$ and $k$ is a consequence of \cite[Proposition 1.2.1]{AsokDasFasel2021} replacing appeal to \cite[Theorem 1.1.1]{AsokDasFasel2021} by appeal to Corollary~\ref{cor:affinerepresentability}.  The second statement then follows from \cite[Theorem 1.3.4]{AsokDasFasel2021}.
    
    For the third statement, we appeal to \cite[Theorem 3.1.13 and Remark 3.1.14]{AsokDasFasel2021}; the latter statement explains exactly how characteristic hypotheses enter the story.  The last statement then follows from \cite[Theorem 3.2.1]{AsokDasFasel2021} 
\end{proof}

Finally, we can also make some statements about weak Euler class groups.  Assume $k$ is an infinite field, and $X$ is a smooth affine $k$-scheme of dimension $d \geq 2$.  Let $Z_0(X)$ be the group of zero cycles on $X$ and $CI_0(X)$ the subgroup generated by reduced complete intersection ideals in $X$.  The quotient 
\[
\op{E}_0(X) := Z_0(X)/CI_0(X)
\] 
is usually known as the weak Euler class group after the work of Bhatwadekar and Sridharan.  There is a well-defined surjective homomorphism
\[
\op{E}_0(X) \longrightarrow CH_0(X)
\]
by \cite[Lemma 2.5]{BhatwadekarSridharan}.  

\begin{thm}[Asok, Fasel]
If $k$ is an infinite field and $X$ is a smooth affine $k$-scheme of dimension $d \geq 2$, then the map
\[
\op{E}_0(X) \longrightarrow CH_0(X)
\]
is an isomorphism.
\end{thm}

\begin{proof}
Repeat the proof of \cite[Theorem 3.2.6]{AsokDasFasel2021} appealing to Theorem~\ref{thmintro:cohomotopy} instead of \cite[Theorem 3.2.1]{AsokDasFasel2021}.
\end{proof}

{\begin{footnotesize}
\raggedright
\bibliographystyle{alpha}
\bibliography{projectivespaces}
\end{footnotesize}}

\Addresses

\end{document}